\documentclass[12pt,reqno,fleqn]{amsart}  
\usepackage{tikz}
\usepackage[top=1.5in, bottom=1.5in, left=1.25in, right=1.25in]	{geometry}
\usepackage{amsmath,amstext,amsthm,amssymb,amsxtra}
\usepackage[T1]{fontenc}

\usepackage{euscript}
\usepackage{mathrsfs}

\usepackage{mathtools}
\mathtoolsset{showonlyrefs,showmanualtags}

\usepackage{hyperref} 
\hypersetup{
	colorlinks=true,       
	linkcolor=blue,          
	citecolor=magenta,        
	filecolor=magenta,      
	urlcolor=cyan           
}

\usepackage[msc-links]{amsrefs}

\theoremstyle{plain} 
\newtheorem{lemma}{Lemma}[section]
\newtheorem{proposition}{Proposition}[section] 
\newtheorem{theorem}{Theorem} [section]
\newtheorem{corollary}{Corollary} [section]

\newtheorem{OldTheorem}{Theorem}

\theoremstyle{remark}
\newtheorem{remark}{Remark}

\def\x{\ensuremath{\textbf x}}

\def\n{\ensuremath{\textbf n}}

\def\spec{{\rm spec}}

\def\ZB{\ensuremath{\mathscr B}}

\def\zB{\ensuremath{\mathbb B}}
\def\ZA{\ensuremath{\mathscr A}}
\def\ZZ{\ensuremath{\mathbb Z}}
\def\ZI{\ensuremath{\mathbf 1}}
\def\ZN{\ensuremath{\mathbb N}}

\def\kappa{\ensuremath{\mathcal K}}
\def\ZR{\ensuremath{\mathbb R}}

\def\ZD{\ensuremath{\mathscr D}}

\def\ZT{\ensuremath{\mathbb T}}
\def\zT{\ensuremath{\mathscr T}}

\theoremstyle{plain}
\newtheorem{definition}{Definition}[section]

\numberwithin{equation}{section}

\def\md#1#2\emd{\ifx0#1
	\begin{equation*} #2 \end{equation*}\fi  
	\ifx1#1\begin{equation}#2\end{equation}\fi   
	\ifx2#1\begin{align*}#2\end{align*}\fi   
	\ifx3#1\begin{align}#2\end{align}\fi    
	\ifx4#1\begin{gather*}#2\end{gather*}\fi  
	\ifx5#1\begin{gather}#2\end{gather}\fi   
	\ifx6#1\begin{multline*}#2\end{multline*}\fi  
	\ifx7#1\begin{multline}#2\end{multline}\fi  
	\ifx8#1\begin{multline*}\begin{split}#2\end{split}\end{multline*}\fi
	\ifx9#1\begin{multline}\begin{split}#2\end{split}\end{multline}\fi
}

\newcommand {\e }[1]{\eqref{#1}}
\newcommand {\lem }[1]{Lemma \ref{#1}}

\newcommand {\pro }[1]{Proposition \ref{#1}}
\newcommand {\trm }[1]{Theorem \ref{#1}}

\title[On UC-multipliers for multiple trigonometric systems] {On UC-multipliers for multiple trigonometric systems }
\author{Grigori A. Karagulyan}
\address{Institute of Mathematics NAS RA, Marshal Baghramian ave., 24/5, Yerevan, 0019, Armenia}
\email{g.karagulyan@gmail.com}

\thanks{Research was supported by the Higher Education and Science Committee of RA, in the frames of the
	research project 21AG-1A045}

\begin{document}
	\maketitle
	\begin{abstract}
		We investigate the class of sequences $w(n)$ that can serve as almost-everywhere convergence Weyl multipliers for all rearrangements of multiple trigonometric systems. We show that any such sequence must satisfy the bounds $\log n\lesssim w(n)\lesssim\log^2 n$. Our main result establishes a general equivalence principle between one-dimensional and multidimensional trigonometric systems, which allows one to extend certain estimates known for the one-dimensional case to higher dimensions.
	\end{abstract}

\section{Introduction}

\subsection{Weyl multipliers}\label{S7}
Let 
\begin{equation}\label{r39}
	\Phi=\{\phi_n:\, n=1,2,\ldots\}\subset L^2(0,1)
\end{equation}
be an orthonormal system. Recall that a sequence of positive numbers $w(n)\nearrow\infty$ is called an almost everywhere convergence Weyl multiplier (or simply a Weyl multiplier)  for $\Phi$ if every series 
\begin{equation}\label{a1}
	\sum_{n=1}^\infty a_n\phi_n(x),
\end{equation}
whose coefficients satisfy 
\begin{equation}\label{a3}
	\sum_{n=1}^\infty a_n^2w(n)<\infty
\end{equation}
converges almost everywhere (see \cite{KaSa} or \cite{KaSt}). 
The classical Menshov-Rademacher theorem (\cite{Men},  \cite{Rad}) asserts that the sequence $\log^2 n$ is a Weyl multiplier for every orthonormal system.  The optimality of $\log^2 n$ of this sequence was also shown by Menshov in \cite{Men}, where he constructed an orthonormal system for which any sequence $w(n)=o(\log^2n)$ fails to be Weyl multiplier. The following concepts are standard in the theory of orthogonal series
\begin{definition}
	A sequence of positive numbers $w(n)\nearrow\infty$ is called an almost everywhere convergence Weyl multiplier for rearrangements (RC-multiplier) of an orthonormal system $\Phi=\{\phi_n\}_{n\ge 1}$ if it is a Weyl multiplier for every system $\{\phi_{n_k}\}$, where $\{n_k\}$ is a sequence of distinct  naturals (not necessarily increasing).
\end{definition}
\begin{definition}
	A sequence of positive numbers $w(n)\nearrow\infty$ is called an almost everywhere unconditional convergence Weyl multiplier (UC-multiplier) for an orthonormal system $\Phi=\{\phi_n\}_{n\ge 1}$ if  under condition \e{a3} series \e{a1} converges almost everywhere after every rearrangement of its terms.
\end{definition}
By the Menshov–Rademacher theorem, the sequence $\log^2 n$ is RC-multiplier for any orthonormal system $\Phi$ and Menshov's counterexample shows that $\log^2 n$ is optimal in this sense. The study of RC- and UC-multipliers for classical orthonormal systems is a longstanding topic in the theory of orthogonal series. It is well known that the constant sequence $w(n)\equiv 1$ is a Weyl multiplier for trigonometric, Walsh, Haar and Franklin systems, yet it is not an RC-multiplier for any of them. Kolmogorov \cite{Kol} was the first to observe that the sequence $w(n)\equiv 1$ is not RC-multiplier for the trigonometric system. However, he has never published the proof of this fact. A proof of this assertion was later given by Zahorski \cite{Zag}. Afterward, developing Zahorski's argument, Ul\cprime yanov \cite{Uly6, Uly7} established such a property for Haar and Walsh systems. Using the Haar functions technique, Olevskii \cite{Ole} succeed proving that such a phenomenon is common for arbitrary complete orthonormal system.
Ul\cprime yanov \cite{Uly1, Uly4} found the optimal growth of the RC and UC multipliers of Haar system. Moreover, his  technique of the proof became a key argument in the study of the analogous problems for other classical systems.

\begin{OldTheorem}[Ul\cprime yanov, \cite{Uly4}]\label{OT1}
	The sequence $\log n$ is an RC-multiplier for the Haar system.
\end{OldTheorem}
\begin{OldTheorem}[Ul\cprime yanov, \cite{Uly4}]\label{OT2}
	A non-decreasing sequence $w(n)$ is a UC-multiplier for the Haar system if and only if
	\begin{equation}\label{a4}
		\sum_{n=1}^\infty\frac{1}{nw(n)}<\infty.
	\end{equation}
\end{OldTheorem}
In fact, if $\log n$ is a RC-multiplier for an orthonormal system, then any non-decreasing sequence $w(n)$ satisfying \e{a4} is UC-multiplier for the same orthonormal system. This follows from a general result that is established in later papers \cite{Uly3, Pol}.

The results of Theorems \ref{OT1} and \ref{OT2} have extensions for some other orthonormal systems of wavelet type. In particular,
G.Gevorkyan in \cite{Gev,Gev1} established the analogue of Theorem \ref{OT2} for the Franklin and Ciesielski systems of wavelet type. For arbitrary wavelet type systems the similar results were recently proved by the author in \cite{KaKa}. Moreover, paper \cite{KaKa} provides a new approach that enables one to get  analogues of both Theorems \ref{OT1} and \ref{OT2} for systems of non-overlapping polynomials with respect to general  wavelet type systems (some details of the papers is given below).

For character type orthogonal systems, such as the trigonometric and Walsh systems, the problem of characterization of RC and UC multipliers is not completely solved yet. These problems were first posed in Ul\'yanov’s 1964 survey \cite{Uly5} and revisited in several later papers (see \cite{Uly4}, p. 1041, \cite{Uly8}, p. 80, \cite{Uly9}, p. 57). The Menshov-Rademacher theorem implies that $\log^2 n$ is a RC-multiplier for both the trigonometric and Walsh systems, and no sequence $w(n)=o(\log^2n)$ is known to be RC-multiplier for either system. 
Lower bounds for the RC and UC multipliers of the Walsh system were studied in \cite{Boch, Nak5, Nak3, Tan3}. The best result, proved independently by Bochkarev \cite{Boch2, Boch} and Nakata \cite{Nak3}, shows that if an increasing sequence $w(n)$ fails to satisfy Ul\'yanov's condition \e{a4},
then it is not a UC-multiplier for the Walsh system. For the trigonometric system the lower bounds for the RC and UC multipliers were studied in \cite{Mor, Tan2, Nak1, Nak2, Seroj,Kar2,Kar3}. The analogous of the Bochkarev-Nakata \cite{Boch, Nak3} theorem for the trigonometric system was recently proved by the author in \cite{Kar2}. We also mention a recent paper \cite{Kar5}, where it was proved that for dyadic orthonormal trigonometric polynomials the sequence $\sqrt{\log n}$ is RC-multiplier.

For a fixed orthonormal system \e{r39} one can consider orthonormal systems formed by non-overlapping polynomials
\begin{equation}\label{r40}
	p_n(x)=\sum_{k\in G_k}a_k\phi_k(x),\quad k=1,2,\ldots,
\end{equation}
i.e. $G_k$ are pairwise disjoint sets of natural numbers. 

\begin{definition}
	A sequence of positive numbers $w(n)\nearrow\infty$ is said to be a strong RC-multiplier (SRC-multiplier) for an orthonormal system \e{r39} if it is a  Weyl multiplier for any orthonormal system of non-overlapping polynomials \e{r40}.
\end{definition}
Observe that if $w(n)$ is SRC-multiplier of an orthonormal system $\Phi$, then it is also RC-multiplier for $\Phi$, since any subsequence $\{\phi_{n_k}\}$ corresponds to  \e{r40} with each $G_k=\{n_k\}$. 
It is of recent interest finding the best SRC-multiplier for a given orthonormal system. It was proved in \cite{Kar1} that for the general martingale difference systems the optimal SRC-multiplier is the sequence $\log n$. An analogous result for general wavelet-type systems was proved in \cite{KaKa}.
Hence, we have $\log n$ is a Weyl multiplier for any orthonormal system of non-overlapping polynomials with respect Haar, Franklin and all other wavelet type systems (see definition in \cite{KaKa}). The optimal growth of a sequence $w(n)$ that can serve as an RC or SRS multiplier for the trigonometric or Walsh systems is still unknown. From the Menshov–Rademacher theorem together with the results of \cite{Boch, Nak3, Kar5}, it follows only that any such sequence must satisfy 
\begin{equation}\label{x24}
	\log n\lesssim w(n)\lesssim \log^2 n. 
\end{equation} 

In this paper, we study the RC(SRS)-multiplier problem for the multiple trigonometric system. Namely, we prove that this problem coincides with the corresponding problem for the one-dimensional trigonometric system.
\begin{theorem}\label{T1}
	Let $w(n)$ be an increasing numerical sequence satisfying $w(n^2)\le Cw(n)$. Then $w(n)$ is an RC(SRS)-multiplier for the one-dimensional trigonometric system if and only if it is an RC(SRS)-multiplier for the multiple trigonometric system.
\end{theorem}
\begin{remark}
	Note that the bound $w(n^2)\le Cw(n)$ is natural, since Weyl multipliers are known to have at most logarithmic growth, and this inequality is satisfied for any sequence $w(n)$ with logarithmic behavior.  
\end{remark}
\begin{remark}
	Observe that for the Walsh system the result of \trm{T1} is immediate, since the Walsh system is probabilistically equivalent to its multiple system. Indeed, the functions in both systems are generated by all possible products of independent random variables taking the values $\pm1$ with equal probability. Such an equivalence does not hold for the trigonometric system. Nevertheless, in this paper we develop a technique that effectively replaces this equivalence principle in the trigonometric setting.
\end{remark}
Applying \trm{T1}, one can extend the bound \e{x24}, known for the one-dimensional trigonometric system, to the multidimensional case.
\begin{corollary}
	If a sequence $w(n)$ is a RC-multiplier of a multiple trigonometric system, then $\log n\lesssim w(n)\lesssim \log^2 n$.  
\end{corollary}

\section{Probabilistically equivalence and measure-preserving maps}
\begin{definition}\label{D1}
	Let $(X,\ZA,\mu)$ and $(Y,\ZB,\nu)$ be probability spaces. Sequences of measurable functions $f_k:X\to \ZR$ and $g_k:Y\to \ZR$, $k=1,2,\ldots$, are said to be probabilistically equivalent if they share the same joint cumulative distribution function, i.e. we have
	\begin{equation}\label{b8}
		\mu\left(\bigcap_{k=1}^nf_k^{-1}(E_k)\right)=\nu\left(\bigcap_{k=1}^ng_k^{-1}(E_k)\right)
	\end{equation}
	for every choice of Borel sets $E_k\subset \zB$, $k=1,2,\ldots,n$. 
\end{definition}
A natural way to obtain probabilistically equivalent sequences is to use a measure-preserving map, as shown in \lem{L6} below.
\begin{definition}
	Let $(X,\ZA,\mu)$ and $(Y,\ZB,\nu)$ be probability spaces. We say a function $\Theta:\ZA \to \ZB$ is measure-preserving map (MP-map), if for any measurable sets $E,F\in \ZA$ we have
	\begin{align}
		&\text{1) }\mu(E)=\nu(\Theta(E)),\\
		&\text{2) }\mu(E\cup F)=\nu(\Theta(E)\cup \Theta(F)).\label{x43}
	\end{align}
\end{definition}

\begin{lemma}\label{L1}
	If  $\Theta$ is a MP-map  from a probability space $(X,\ZA,\mu)$ to $(Y,\ZB,\nu)$, then for any measurable sets $E,F\in \ZA$, $E_k\in \ZA$, $k=1,2,\ldots$,
	\begin{align}
		&\mu(E\setminus F)=\nu(\Theta(E)\setminus \Theta(F)),\label{x6}\\
		&\mu\left(\bigcup_k E_k\right)=\nu\left(\bigcup_k \Theta(E_k)\right),\label{x42}\\
		&\mu\left(\bigcap_k E_k\right)=\nu\left(\bigcap_k \Theta(E_k)\right).\label{a0}
	\end{align}
\end{lemma}
\begin{proof}
	Using the definition of MP-map, we obtain
	\begin{align*}
		\mu(F)+\mu(E\setminus F)&=\mu(F\cup E)=\nu(\Theta(F)\cup\Theta(E))\\
		&=\nu(\Theta(F))+\nu(\Theta(E)\setminus\Theta(F))\\
		&=\mu(F)+\nu(\Theta(E)\setminus\Theta(F)).
	\end{align*}
	Thus \e{x6} follows. Similarly,
	\begin{align*}
		\mu(E)+\mu(F)-\mu(E\cap F)&=\mu(E\cup F)=\nu(\Theta(E)\cup \Theta(F)))\\
		&=\nu(\Theta(E))+ \nu(\Theta(F)))-\nu(\Theta(E)\cap \Theta(F)))\\
		&=\mu(E)+\mu(F)-\nu(\Theta(E)\cap \Theta(F))).
	\end{align*}
	This implies \e{a0} for two sets and so for a finite collection of measurable sets. Then passing to a limit we will get \e{a0} for any countable collection of sets. Similarly \e{x43} implies \e{x42}. 
\end{proof}

Let $\Theta$ be a MP-map from a probability space $(X,\ZA,\mu)$ to $(Y,\ZB,\nu)$. Then $\Theta$ induces a map that takes any indicator function $\ZI_E$, $E\in \ZA$, to $\ZI_{\Theta(E)}$. The following lemma gives a suitable extension of $\Theta$ to entire space $L^0(X)$ of measurable functions. 
\begin{lemma}\label{L6}
	If  $\Theta$ is a MP-map  from a probability space $(X,\ZA,\mu)$ to a probability space $(Y,\ZB,\nu)$, then it can be uniquely extended as an operator from $L^0(X)$ into $L^0(Y)$ such that  every sequence $\{f_k\}_{k\ge1}\subset L^0(X)$ is probabilistically equivalent to its image sequence $\{\Theta(f_k)\}_{k\ge 1}$.
\end{lemma}
\begin{proof}
	Let $s(x)$ be a simple function, i.e.
	\begin{equation}
		s(x)=\sum_{i=1}^m\alpha_i\ZI_{E_i}(x),
	\end{equation}
	where $E_i\in \ZA$, $i=1,2,\ldots,m $, are pairwise disjoint measurable sets and $\alpha_i\in \zB$. Then we define 
	\begin{equation}
		\Theta s(x)=\sum_{i=1}^m\alpha_i\ZI_{\Theta(E_i)}(x).
	\end{equation}
	For an arbitrary  $f\in L^0_\zB(X)$, one can find a sequence of simple functions $s_n(x)$, such that 
	\begin{equation}\label{a2}
		f(x)=\lim_{n\to\infty}s_n(x)\text{ a.e..}
	\end{equation}
	Using \lem{L1} and a standard argument one can observe that the sequences $s_n(x)$ and $\Theta s_n(x)$ are probabilistically equivalent. Consequently, $\Theta s_n(x)$ converges almost everywhere and the limit doesn't depend on the particular choice of simple functions $s_n$ in \e{a2}. We therefore define this limit as $\Theta f(x)$. It is also clear that any sequence $\{f_k\}_{k\ge1}\subset L^0_\zB(X)$ is probabilistically equivalent to its image sequence $\{\Theta(f_k)\}_{k\ge1}$.
\end{proof}

\section{A decomposition for discrete trigonometric systems}
For an integer $l\ge 1$ we denote $\ZN_l=\{0,1,\ldots ,l-1\}$ and let 
\begin{equation*}
	\ZN_{p_1,\ldots,p_d}= \ZN_{p_1}\times\cdots\times\ZN_{p_d}
\end{equation*}
be the Cartesian product of the sets $\ZN_{p_k}$, $k=1,\ldots, d$. 
The discrete trigonometric system (DTS) of order $l$ is the collection of $l$ orthonormal functions on $\ZT=\ZR/\ZZ$ defined by
\begin{equation}\label{x18}
	\ZD^{(l)}=\left\{t_n^{(l)}(x)=\sum_{k=0}^{l-1}\exp\left(2\pi i \frac{nk}{l}\right)\cdot \ZI_{\delta_k^{(l)}}(x):\,\quad n\in \ZN_l\right\},
\end{equation}
where $\delta_k^{(l)}=[k/l,(k+1)/l)$. We consider those as $1$-periodic functions on $\ZR$.  For integers $p_k\ge 2$, $k=1,2,\ldots,d$, consider the multiple DTS $\ZD^{(p_1,\ldots,p_d)}$, which is the tensor product of the one dimensional systems $\ZD^{(p_k)}$. This system consists of the functions
\begin{align*}
t_{\n}^{(p_1,\ldots,p_d)}(\x)&=\prod_{k=1}^dt_{n_k}^{(p_k)}(x_k)\\
&=\sum_{u_1=1}^{p_1}\cdots \sum_{u_d=1}^{p_d}\exp2\pi i\left( \frac{n_1u_1}{p_1}+\cdots +\frac{n_du_d}{p_d}\right)\cdot \ZI_{\delta_{u_1}^{(p_1)}\times \cdots\times \delta_{u_d}^{(p_d)}}(\x),\\
&\x=(x_1,\ldots,x_d)\in \ZT^d,\quad \n=(n_1,\ldots,n_d)\in \ZN_{p_1,\ldots,p_d},
\end{align*}
where $\delta_{u_1}^{(p_1)}\times \cdots\times \delta_{u_d}^{(p_d)}$ is the Cartesian product of the intervals $\delta_{u_k}^{(p_k)}$. In the sequel the notation $\{a\}$ denotes the fractional part of a number $a$. For positive integers $n$ and $p$ let denote by $\langle n\rangle_p$ the remainder when $n$ is divided by $p$. 
\begin{proposition}\label{P1}
	If $p_1,\ldots, p_d$ are mutually coprime numbers and $p=p_1\ldots p_d$, then the  systems $\ZD^{(p)}$ and $\ZD^{(p_1,\ldots,p_d)}$ are probabilistically equivalent. Moreover, this equivalence is generated by a measure-preserving map $\Theta:\ZT^p\to \ZT$.
\end{proposition}
\begin{proof}
According to the Chinese remainder theorem for any integers  $n_j\in \ZN_{p_j}$, $j=1,2,\ldots,d$ there is a unique integer $l\in \ZN_{p}$ such that
\begin{align*}
	l=n_j\mod p_j,\quad j=1,2,\ldots,d.
\end{align*}
This defines a one-to-one mapping  $\tau:\ZN_{p_1,\ldots,p_d}\to \ZN_{p}$ such that $\tau(n_1,n_2,\ldots, n_d)=l$. Consider the mapping $\bar \tau:\sqcap_{j=1}^d\ZN_{p_j}\to \ZN_{p}$ defined by
\begin{equation*}
	\bar\tau(u_1,\ldots,u_d)=\left\langle\tau(u_1,\ldots,u_d)\sum_{j=1}^d\frac{p}{p_j}\right\rangle_p.
\end{equation*}
Observe that this is a one-to-one mapping. 
Since the sets $\sqcap_{j=1}^d\ZN_{p_j}$ and $\ZN_{p}$ both have cardinality $p$ it remains to show that  $\bar \tau$ is injective. For $ (u_1,\ldots,u_d)\neq (u'_1,\ldots,u'_d)$ we have
	\begin{align}
		&\left\langle\bar \tau (u_1,\ldots,u_d)-\bar \tau(u'_1,\ldots,u'_d)\right\rangle_p\\
		&\qquad\qquad\qquad\qquad =\left\langle	(\tau (u_1,\ldots,u_d)- \tau(u'_1,\ldots,u'_d))\sum_{j=1}^d\frac{p}{p_j}\right\rangle_p.\label{x1}
	\end{align}
	One can check that $\tau (u_1,\ldots,u_d)- \tau(u'_1,\ldots,u'_d)$ is not divisible by $p$, as well as  $\sum_{j=1}^d\frac{p}{p_j}$ and $p$ are coprime numbers. This implies that \e{x1} can not be $0$ and so $\bar \tau$ is injective. Using $\bar\tau$ we define a MP-map $\Theta:\ZT^d\to \ZT$ by choosing 
	\begin{equation*}
		\Theta(\delta_{u_1}^{(p_1)}\times \cdots\times \delta_{u_d}^{(p_d)})=\delta^{(p)}_u \text{ if } u=\bar \tau(u_1,\ldots,u_d),
	\end{equation*}
while inside of each $\delta_{u_1}^{(p_1)}\times \cdots\times \delta_{u_d}^{(p_d)}$ the map $\Theta$ is defined arbitrarily,  provided that the measure-preserving property is maintained. Using the definitions of $\tau$ and $\bar \tau$, we obtain
\begin{align*}
	&\left\{\frac{\tau(n_1,\ldots,n_d)\bar \tau(u_1,\ldots,u_d)}{p}\right\}\\
	&\qquad\qquad\qquad\qquad=	\left\{\frac{\tau(n_1,\ldots,n_d)\tau(u_1,\ldots,u_d)\sum_{j=1}^d\frac{p}{p_j}}{p}\right\}\\
	&\qquad\qquad\qquad\qquad=	\left\{\sum_{j=1}^d\frac{\tau(n_1,\ldots,n_d)\tau(u_1,\ldots,u_d)}{p_j}\right\}=\left\{\sum_{j=1}^d\frac{n_ju_j}{p_j}\right\}.
\end{align*}
Thus for $n=\tau(n_1,\ldots,n_d)$ and $u=\bar \tau (u_1,\ldots,u_d)$ we have
\begin{align*}
	\exp\left(2\pi i\cdot  \frac{nu}{p}\right)&=\exp\left(2\pi i\cdot  \frac{\bar\tau(n_1,\ldots,n_d) \tau(u_1,\ldots,u_d)}{p}\right)\\
	&=\prod_{j=1}^d\exp\left(2\pi i\cdot  \frac{n_j u_j}{p_j}\right).
\end{align*}
This implies that 
\begin{equation*}
	\Theta(t_{n_1,\ldots,n_d}^{(p_1,\ldots,p_d)})=t^{(p)}_n,\text{ whenever } n=\tau(n_1,\ldots,n_d). 
\end{equation*}
Therefore the systems $\ZD^{(p)}$ and $\ZD^{(p_1,\ldots,p_d)}$ are probabilistically equivalent.
\end{proof}
\section{Auxiliary lemmas and the proof of \trm{T1}}
Denote by $\zT_d$ the $d$-dimensional  trigonometric system, which consists of the functions
\begin{align}
	&t_{\n}(\x)=\exp(2\pi i(n_1x_1+\cdots +n_dx_d)),\text{ where }\\
	&\x=(x_1,\ldots,x_d)\in \ZT^d,  \quad \n=(n_1,\ldots,n_d)\in \ZZ^d.
\end{align}
The one-dimensional trigonometric system (when $d=1$) will be simply denoted by $\zT$.
\begin{definition}
	Let $(X,\ZA,\mu)$ and $(Y,\ZB,\nu)$ be probability spaces. Sequences of measurable functions $f_k:X\to \ZR$ and $g_k:Y\to \ZR$, $k=1,2,\ldots,n$, are said to be probabilistically equivalent with an error $\varepsilon>0$ if there is a sequence $\bar g_k:Y\to \ZR$, $k=1,2,\ldots,n$, probabilistically equivalent to $\{f_k\}$ such that $\|\bar g_k-g_k\|_2<\varepsilon$, $k=1,2,\ldots,n$. 
\end{definition}
\begin{lemma}\label{L3}
Let $f_k\in\zT_d$, $k=1,2,\ldots,n$ be an arbitrary choice of $d$-dimensional trigonometric functions. Then for any $\varepsilon>0$ there exist an integer $p$ and a sequence of discrete one-dimensional trigonometric functions $g_k\in  \ZD^{(p)}$, $k=1,2,\ldots,n$,  such that the sequences $\{f_k\}$ and $\{g_k\}$ are probabilistically equivalent with an error $\varepsilon$.
\end{lemma}
\begin{proof}
	 Let
	\begin{equation}\label{x5}
		f_k(\x)=t_{n_1^k,\ldots,n_d^k}(\x),\quad k=1,2,\ldots,n.
	\end{equation}
We will find $p$ in the form $p=p_1\cdots p_d$, where $p_k$ are mutually coprime integers. By choosing the numbers $p_j\ge \max_{1\le k\le n}n_j^k$ large enough, we may achieve a good approximation of functions  \e{x5} by the corresponding discrete trigonometric functions, that is
\begin{equation}\label{x8}
\left|t_{n_1^k,\ldots,n_d^k}(\x)-t^{(p_1,\ldots,p_d)}_{n_1^k,\ldots,n_d^k}(\x)\right|<\varepsilon \text{ for all }k=1,2,\ldots,n, \quad \x\in \ZT^d.
\end{equation}
Then applying \pro{P1}, we can say that the sequence $t^{(p_1,\ldots,p_d)}_{n_1^k,\ldots,n_d^k}$, $k=1,2,\ldots,n$, is probabilistically equivalent to a sequence of discrete one-dimensional trigonometric functions $g_k\in \ZD^{(p)}$. This together with \e{x8} implies that the sequences $\{f_k\}$ and $\{g_k\}$ are probabilistically equivalent with an error $\varepsilon>0$, completing the proof of lemma.
\end{proof}
Let $c_k(f)$, $k\in \ZZ$, be the Fourier coefficients of a function $f\in L^2(\ZT)$ with respect to one dimensional trigonometric system $\zT$ and denote the spectrum of $f$ by
\begin{equation*}
\spec(f)=\{k\in \ZZ:\, c_k(f)\neq 0\}.
\end{equation*}
\begin{lemma}\label{L2}
For every integer $l\ge  1$ the functions $t_n^{(l)}\in \ZD^{(l)}$, $n=0,1,\ldots,l-1$, have non-overlapping spectrums with respect to $\zT$. Moreover,  there are  non-overlapping one-dimensional trigonometric polynomials $g_k$, $k=0,1,\ldots,l-1$, each is a linear combination of $l$ trigonometric functions and 
\begin{equation*}
	\|g_n\|_{L^2(\ZT)}\le 1,\quad \|t_n^{(l)}-g_n\|_{L^2(\ZT)}\lesssim 1/\sqrt{l}, \quad n=0,1,\ldots,l-1.
\end{equation*}
\end{lemma}
\begin{proof}
For the Fourier coefficients of the functions $t_n^{(l)}$ (see \e{x18}) we have
\begin{align}
	c_m(t_n^{(l)})&=\int_0^1t_n^{(l)}(x)\exp\left(-2\pi i mx\right)dx\\
	&=\sum_{k=0}^{l-1}\exp\left(2\pi i \frac{nk}{l}\right)\int_{k/l}^{(k+1)/l}\exp\left(-2\pi i mx\right)dx\\
	&=\exp\left(\frac{2\pi im}{l}\right)\sum_{k=0}^{l-1}\exp\left(2\pi i \frac{(n-m)k}{l}\right)\int_{0}^{1/l}\exp\left(2\pi i mx\right)dx\\
	&=\exp\left(\frac{2\pi im}{l}\right)\frac{\exp(2\pi im/l)-1}{2\pi im}\sum_{k=0}^{l-1}\exp\left(2\pi i \frac{(n-m)k}{l}\right).\label{x19}
\end{align}
If $m\neq n\mod l$ the sum in \e{x19} is zero and therefore $c_m(t_n^{(l)})=0$ . Thus we obtain $\spec(t_n^{(l)})=\{n+lj:\, j\in \ZZ\}$, $n=0,1,\ldots, l-1$, are non-overlapping. On the other hand, choosing $m=n+lj$ in \e{x19} we get 
\begin{equation}\label{x9}
	|c_{n+Nj}(t_n^{(l)})|\lesssim \frac{1}{|j|+1},\quad j\in \ZZ.
\end{equation}
Thus for the polynomials
\begin{equation*}
	g_n(x)=\sum_{|j|< l/2}c_{n+lj}(t_n^{(l)})\exp\left( 2\pi i (n+lj)x\right),\quad n=0,1,\ldots, l-1,
\end{equation*}
we get
\begin{equation*}
	\|t_n^{(l)}-g_n\|_2^2\lesssim \sum_{|j|\ge l/2}\frac{1}{j^2}\lesssim \frac{1}{l}.
\end{equation*}
This completes the proof of the lemma.
\end{proof}
\begin{lemma}\label{L4}
	Let $\{\phi_n(x)\}_{n=1}^\infty$ be an orthonormal system on $(0,1)$ and suppose that $w(n)$ is an increasing sequence satisfying $w(n)\ge c\log n$ for a constant $c>0$. Then $w(n)$ is a Weyl multiplier for  $\{\phi_n(x)\}$ if and only if under condition \e{a3},
\begin{equation}\label{x11}
	\lim_{k\to\infty}\max_{2^k\le m<2^{k+1}}\left|\sum_{n=2^k}^ma_n\phi_n(x)\right|=0\text{ a.e..}
\end{equation}
\end{lemma}
\begin{proof}
	Without loss of generality the condition $w(n)\ge c\log n$ can be replaced by 
	\begin{equation}\label{x10}
		w(n)\ge \log n. 
	\end{equation}
	Hence we can suppose that $w(n)$ satisfies \e{x10}. It is known that under the condition \e{x10} from \e{a3} it follows that
	\begin{equation}\label{x12}
		\lim_{k\to\infty}\sum_{n=1}^{2^{k}}a_n\phi_n(x)\text{ exists a.e..}
	\end{equation}
(see \cite{KaSt}, Lemma [5.3.2]). From \e{x11} and \e{x12} one can get the a.e. convergence of series \e{a1}. Hence, $w(n)$ is a Weyl multiplier for the system $\{\phi_n\}$. Now suppose conversely that
$w(n)$ satisfies \e{x10}  and it is a Weyl multiplier for $\{\phi_n\}$. Then under the condition \e{a3} series \e{a1} converges a.e., which immediately implies \e{x11}.
\end{proof}
\begin{proof}[Proof of \trm{T1}]
	Let $\{f_n\}_{n\ge 1}\subset \zT_d$ be an arbitrary choice of $d$-dimensional trigonometric functions. Applying first \lem{L3}, then \lem{L2}, for every $k$ we find a collection of non-overlapping one-dimensional trigonometric polynomials $ \{\bar f_n:\, 2^k<n\le 2^{k+1}\}$ such that each $\bar f_n$ is a linear combination of $4^k$ trigonometric functions, 
\begin{equation*}
		\|\bar f_n\|_2\le 1, \quad n=1,2,\ldots,
\end{equation*}
and the blocks 
	\begin{equation}\label{x14}
		\{f_n:\, 2^k\le n<2^{k+1}\} \text{ and } \{\bar f_n:\, 2^k\le n <2^{k+1}\},\quad k=0,1,2,\ldots
	\end{equation}
are  probabilistically equivalent with an errors $\lesssim 2^{-k}$. Note that some polynomials in $\{\bar f_n\}_{n\ge1}$ may have overlapping spectrums and it may only happen when those polynomials belong to different blocks. It is clear that there exist integers $n_k$ such that the new polynomial system defined by  
\begin{equation}\label{x13}
	g_n(x)=\bar f_n(x)\cdot e^{2\pi n_kx},\quad 2^k\le n <2^{k+1},
\end{equation}
are non-overlapping. Thus the polynomials $g_n$ can be written in the form 
\begin{equation}\label{x23}
	g_n(x)=\sum_{s=s_n+1}^{s_{n+1}}b_s \exp(2\pi i m_sx),\quad \sum_{s=s_n+1}^{s_{n+1}}b_s^2\le 1.
\end{equation}
where $m_k$, $k=1,2,\ldots$, is a sequence of different integers (not necessarily increasing) such that
\begin{equation}\label{x20}
s_1=0,\quad s_{n+1}-s_n= 4^k\text{ if }2^k\le n<2^{k+1}.
\end{equation}
Observe that from \e{x20} it follows that
\begin{equation}
	s_{n}\le 1+8+8^2+\cdots+8^k\text{ whenever }2^k\le n<2^{k+1}.
\end{equation}
Therefore we have
\begin{equation}\label{x21}
	s_n\le n^4.
\end{equation}
Now let $w(n)$ be a RC-multiplier for the trigonometric system.  Hence, according to the result of  \cite{Kar2} we can suppose that 
\begin{equation}\label{x22}
	w(n)>c\log n. 
\end{equation}
Let us show that $w(n)$ is a Weyl multiplier for $\{g_n\}_{n\ge1}$ as well. Let the a sequence $\{a_n\}$ satisfies \e{a3}.
By \e{x23} we may formally write
\begin{equation}\label{x7}
	\sum_{n=1}^\infty a_ng_n(x)=\sum_{n=1}^{\infty}\sum_{s=s_n+1}^{s_{n+1}}a_nb_s \exp(2\pi i m_sx),
\end{equation}
where the second sum can be considered as a series in a rearranged one-dimensional trigonometric system with the coefficients 
\begin{equation}
	c_s= a_nb_s\text{ if }  s_n<s\le s_{n+1}.
\end{equation}
From \e{x21} and the condition $w(n^2)\le Cw(n)$ it follows that $w(s_{n+1})\le Cw(n)$. Thus we obtain
\begin{align*}
	\sum_{s=1}^\infty c_s^2w(s)&=\sum_{n=1}^\infty a_n^2\sum_{s=s_n+1}^{s_{n+1}}b_s^2w(s)\\
	&\le C\sum_{n=1}^\infty a_n^2w(n)\sum_{s=s_n+1}^{s_{n+1}}b_s^2\\
	&\le C\sum_{n=1}^\infty a_n^2w(n)<\infty
\end{align*}
and therefore, by the assumption that $w(n)$ is a RC-multiplier for the trigonometric system, we get that both series in \e{x7} converge a.e..
Hence, we conclude that $w(n)$ is a Weyl multiplier for $\{g_n\}_{n\ge1}$. Taking into account \e{x22} we can apply \lem{L4}. Then, using also \e{x13}, we obtain
	\begin{equation}\label{x15}
		\lim_{k\to\infty}\max_{2^k\le m<2^{k+1}}\left|\sum_{n=2^k}^{m}a_n\bar f_n(x)\right|=\lim_{k\to\infty}\max_{2^k\le m<2^{k+1}}\left|\sum_{n=2^k}^{m}a_ng_n(x)\right|=0\text{ a.e..}
	\end{equation}
for any sequence $a_n$ satisfying \e{a3}. Since two systems in \e{x14} are probabilistically equivalent with an error $\lesssim 2^{-k}$, there exists an intermediate sequence $\{\bar{\bar f}_n:\, 2^k\le n<2^{k+1}\}\subset L^2(\ZT)$, which is probabilistically equivalent to $ \{f_n:\, 2^k\le n <2^{k+1}\}$ and $\|\bar{\bar f}_n-\bar f_n\|_2\le 2^{-k}$. This implies 
\begin{align}
	\max_{2^k\le m<2^{k+1}}\left|\sum_{n=2^k}^{m}a_n f_n(x)\right|&=\max_{2^k\le m<2^{k+1}}\left|\sum_{n=2^k}^{m}a_n\bar {\bar f}_n(x)\right|\\
	&\le \max_{2^k\le m<2^{k+1}}\left|\sum_{n=2^k}^{m}a_n\bar f_n(x)\right|+\delta_k(x)\label{x16}
\end{align}
where 
\begin{equation}\label{x17}
	\|\delta_k\|_2^2\le \sum_{n=2^k}^{2^{k+1}-1}a_n^2\cdot \sum_{n=2^k}^{2^{k+1}-1}\|\bar f_n-\bar{\bar f}_n\|_2^2\le 2^{-k}\cdot \sum_{n=2^k}^{2^{k+1}-1}a_n^2.
\end{equation}
Hence, it becomes clear that $\delta_k(x)\to 0$ a.e.. Combining this with \e{x15} and \e{x16} , we obtain 
\begin{equation*}
	\lim_{k\to\infty}\max_{2^k\le m<2^{k+1}}\left|\sum_{n=2^k}^{m}a_n f_n(x)\right|=0.
\end{equation*}
Then once again applying \lem{L4} we conclude that $w(n)$ is a Weyl multiplier for the system $\{f_n\}_{n\ge 1}$.

Now consider the part of the theorem concerning the SRC-multipliers. In this case, we use a simplified version of the argument employed in the previous part and omit the details. Let $\{f_n\}_{n\ge 1}$ be an arbitrary orthonormal system of non-overlapping $d$-dimensional trigonometric polynomials.  Applying Lemmas \ref{L3} and \ref{L2}, for every $k$ we find a collection of non-overlapping one-dimensional trigonometric polynomials $ \{\bar f_n:\, 2^k\le n< 2^{k+1}\}$ such that $\|\bar f_n\|_2\le 1$,  $n=1,2,\ldots,$ and the systems \e{x14}
are  probabilistically equivalent with an errors $\lesssim 2^{-k}$. Note that in this case the number of representation terms in the polynomials $\bar f_n$ is not important.
Then similarly we define the sequence \e{x13}. If $w(n)$ is a SRC-multiplier for the trigonometric system, then it is Weyl multiplier for the polynomial system $\{g_n\}_{n\ge 1}$. Thus, applying \lem{L4}, we can immediately write \e{x15}. Then continuing exactly the same procedure after \e{x15} we conclude that $w(n)$ is a Weyl multiplier for our arbitrarily chosen non-overlapping polynomial system $\{f_n\}_{n\ge 1}$. This completes the proof of the theorem.
\end{proof}

\end{document}